\documentclass[12pt,a4]{amsart}

\usepackage{amsmath,amssymb,url}

\newcommand{\xc}[1]{}

\usepackage{xspace}
\usepackage{hyperref}

\newcommand{\IR}{\mathbb{R}}

\newcommand{\IC}{\mathbb{C}}

\newcommand{\IF}{\mathbb{F}}

\newcommand{\QL}{{\sf QL}}

\newcommand{\lat}{{\sf L}}

\newtheorem{theorem}{Theorem}

\begin{document}

\title{Test sets for tautologies in modular quantum logic}
\author{Christian Herrmann}
\address{C. Herrmann, Technische Universit\"{a}t Darmstadt FB4, Schlo{\ss}gartenstr. 7, 64289 Darmstadt, Germany}
\email{herrmann@mathematik.tu-darmstadt.de}

\date{}

\maketitle


\begin{abstract}
As defined by Dunn, Moss, and Wang,
an universal test set in an ortholattice $L$
is a subset $T$ such that each term takes value $1$, only,
if it does so under all substitutions from $T$. 
Generalizing their result for ortholattices of
subspaces of finite dimensional Hilbert spaces,
we show
 that no infinite modular ortholattice of finite dimension
admits a  finite universal test
set. On the other hand, answering a question of the same
authors,
 we  provide a countable universal test set
for the ortholattice of projections of any
type II$_1$ von Neumann algebra factor as well as
for von Neumann's 
algebraic construction of a continuous geometry.
These universal test sets consist of
elements having rational normalized dimension with
denominator a power of $2$.
\end{abstract}

\section{Introduction}
In their seminal paper, 
Birkhoff and von Neumann \cite{BirkhoffVonNeumann},
suggested  ortholattices $\lat(H)$ of closed subspaces
of Hilbert spaces $H$ as structures to 
deal with uncertainity features of quantum mechanics,
the set $\QL(L)$ of \emph{tautologies} of an ortholattice $L$
consisting of all terms $t(\bar x)$ such that
$t(\bar a)=1$ for every substitution in $L$.
In particular, towards  an axiomatization,
they showed that, for fixed finite  $d\geq 4$,
the ortholattices $\lat(H)$ of subspaces of $d$-dimensional
anisotropic inner products spaces  $H$ of $\dim H=d$
are, up to isomorphism, the directly irreducible
$d$-dimensional ortholattices satisfying Dedekind's modular law.
Concerning infinite dimenson,
von Neumann suggested
to consider  certain continuous geometries. These are modular, too.

In contrast, ortholattices
of all subspaces of infinite dimensional Hilbert spaces
satisfy the orthomodular law, only, which became
  focus of Quantum Logic.
Interest in the modular case  was renewed
in
the ``third life'' of quantum logic, inspired by quantum computing:
In \cite{Hagge1}, Dunn, Hagge, Moss, and Wang 
discussed 
quantum logic tautologies  of $\lat(H)$ where 
$H$ is a finite dimensional Hilbert space.
 They derived 
decidability of tautologies, in fixed dimension,
from Tarski's decision procedure
for elementary algebra and geometry
and discussed how tautologies relate to dimensions.
In the abstract setting  of  modular ortholattices
the latter question has been studied by Giuntini, Freytes, and
Sergioli \cite{frey}.

In \cite{dunn}, Dunn, Moss, and Wang introduced
the concept of a
\emph{universal  test set} for an ortholattice $L$:
A subset $T$ of $L$ such that
$t(\bar x) \in \QL(L)$
if and only if $t(\bar a)=1$ for every substitution
$\bar a$ within $T$. They showed that
$\lat(H)$  does not admit a finite universal test set
if $2 \leq \dim H <\infty$ and posed the problem
to establish a universal test set
 for the  ortholattice of projections of the hyperfinite
type II$_1$ von Neumann algebra factor.

The purpose of the present note is
to show non-existence of finite universal test
sets for any infinite modular ortholattice of finite dimension
and to provide countable  universal test sets
for certain continuous geometries:
The ortholattices of projections of 
type II$_1$ von Neumann algebra factors and
 von Neumann's \cite{neu3}
algebraic construction of a continuous geometry
(which is not isomorphic to any of the former,
according to von Neumann \cite{neu2}).
These universal test sets consist of
elements having rational normalized dimension with
denominator a power of $2$.
Observe that $\QL(L)$ is decidable for each of these
continuous geometries $L$ \cite{hn,hard}.

\section{Test sets in MOLs of finite dimension}
An  \emph{ortholattice}
is a  lattice  with bounds $0,1$ and an orthocomplementation
$x \mapsto x'$, that is an order reversing involution
such that $x\wedge x'=0$ and $x \vee x'=1$ (cf.
Section 1.2 of \cite{dunn}).
A lattice or ortholattice $L$ 
 is \emph{modular} if  $x \geq z$ implies $x\wedge(y \vee z)
= (x \wedge y) \vee z$ .
Such  $L$ is of finite \emph{dimension}
 $d$ if some/any maximal chain in $L$ has $d+1$ elements,
we write $d=d(L)$.
We also use MOL for ''modular ortholattice''. 
Any interval $[b,c]$ of a MOL $L$ is
a MOL with the \emph{relative orthocomplement}
$x \mapsto (x'\wedge c)\vee b$.
Observe that $[b,c]$ is isomorphic to $[0,a]$, where
$a= b'\wedge c$, and that the latter is a homomorphic image of the
sub-ortholattice $[0,a] \cup [a',1]$ of $L$. Thus,
any tautology of $L$ is also one of $[b,c]$.
Also recall that with any ortholattice idenity $t_1=t_2$
one can associate a term $t$ such that, within any MOL,
$t_1=t_2$ is equivalent to $t=1$.

\begin{theorem}
An infinite MOL of finite dimension does
not admit a finite universal test set.
\end{theorem}

\begin{proof}
First, we
consider a directly  irreducible MOL  $L$ of $\dim L=d <\infty$.
Recall that the lattice $L$ is isomorphic to the
subspace lattice of an irreducible $d-1$-dimensional
projective space.
In particular, $L$ is finite if so is some of its
intervals $[0,a]$ of dimension $2$.
 Coordinate systems, that is
$d+1$ points in general position, correspond
to non-trival $d$-frames in $L$.

 Here, a
$d$-\emph{frame} $\bar a$ in a modular lattice $M$
is given by $a_0,\ldots ,a_d,a_\bot,a_\top$ in $M$
such that $a_\top=\bigvee_{i=0}^d a_i$ and
$a_\bot =a_j\wedge \bigvee_{i\neq j} a_i$ for all $j$.
A $d$-frame $\bar a$ is \emph{trivial} if $a_\bot=a_\top$;
otherwise, any set $\{a_i\mid i\neq j\}$ is independent.
In particular, any $d$-frame in $M$ of $\dim M<d$ has
to be trivial.

 Huhn \cite{huhn}
has provided a tuplet $\bar a^d(\bar z^d)$ of lattice terms
which yields a $d$-frame $\bar a^d(\bar b)$
 for every substituion $\bar b$ in a modular lattice and
such that $\bar a^d(\bar b)=\bar b$ if $\bar b$ is a $d$-frame.
For the equivalent concept of von Neumann normalized frames
of order $d$ the analoguous result has been obtained by Freese \cite{freese}.

Now, observe that there is an ortholattice term
$s(y_0,y_1,y_2,y_3)$  such that for any $c_1,c_2 \in [c_0,c_3]$
in a MOL one has $t(\bar c)$
a complement of $c_2$ in $[c_0,c_3]$
and such that $s(\bar c)=c_1$ if, in addition, 
$c_1$ is a complement of $c_2$ in $[c_0,c_3]$.
Namely, choosing $d_1$ the relative orthocomplement 
of $c_1 \wedge c_2$ in $[c_0,c_1]$ 
one obtains a complement of $c_2$ in $[c_0,c_3]$
by  the relative orthocomplement of $c_1\vee c_2$
in $[d_1, c_3]$.

Finally, put $\bar x_n=(x_1,\ldots,x_n)$, where $x_1,x_2, \ldots$
 are  pairwise distinct variables,
 and define
\[ \hat{x}^d_i:=  s(z^d_\bot, (x_i\wedge(z^d_0\vee z^d_1))\vee z_\bot, z^d_1,z^d_0\vee z^d_1)
\]
\[t^d_n(\bar z,\bar x_n):= (a^d_\top)'\vee \bigvee_{k=1}^{d-1}z^d_k \vee  \bigvee_{1\leq i<j\leq n}
(\hat{x}^d_i \wedge \hat{x}^d_j).   \]

Observe that $z^d_\bot \leq \hat{x}^d_i \leq z^d_0\vee z^d_1$
holds in any MOL $M$. Thus, 
given a substitution $\bar b,\bar c$
one has $t^d_n(\bar b,\bar c)=1$ if the $d$-frame $\bar a:=\bar a^d(\bar b)$
is trivial. Otherwise,
 if in addition $M$ is directly irreducible and $\dim M=d$, 
then $a_\bot=0$, $a_\top=1$, and
 the $a_k$ and $p_i:=\hat{x}^d_i(\bar b,\bar c)$ are atoms of $M$,
the latter also complements of $a_1$ in $[0,a_0\vee a_1]$.
Therefore,  $t^d_n(\bar b,\bar c)=1$ if and only if
$p_i=p_j$ for some $i<j$. It  follows that  $t^d_n(\bar b,\bar c)=1$
for all substitutions in a set $T$ with $|T|<n$.
On the other hand, if $M$ is infinite  
then there is  a substitution such that
$t^d_n(\bar b,\bar c)<1$, namely $\bar b$ some non-trivial $d$-frame
and the $c_i$ pairwise distinct atoms in $[0,b_0\vee b_1]$,
$c_i \neq b_1$.

Now, consider any infinite MOL $L$
of finite dimension.
 Then, up to isomorphism,
$L$ is a direct product of finitely many directly irreducibles.
Let $L_0$ be a factor of maximal dimension $d$.
One has $d \geq 2$ since  $d=1$ would imply
that  $L$ is distributive and finite.
If $d=2$, all factors are of dimension $\leq 2$ and
we my assume $L_0$ infinite.  If $d>2$  
then $L_0$ is infinite due to \cite{baer}.
As observed, above,  for each $n$ there is an assignment in $L_0$,
such that $t^d_n$ does not evaluate to $1$. It follows
that $t^d_n$ is not a tautology of $L$. Now, given any finite
$T \subseteq L$, $t^d_n$ with $n>|T|$ witnesses that
$T$ is not a universal test set.
\end{proof}

Within modular lattices, the identity $a^d_\bot=a^d_\top$
is equivalent to Huhn's $d-1$-distributive law.
Given $n<\infty$, with
 \cite[Theorem 2.12]{jon} 
it follows that a subdirectly irreducible MOL $L$ is
 of $\dim L\leq n$   if and only if
$L$ is $n$-distributive -- whence  of $\dim L =n$ if and only if,
in addition, $L$ is not $n-1$-distributive.
In particular, this applies 
to directly irreducible complete  atomic  MOLs,
cf. \cite[Theorem 4.10]{frey}.

\section{Test sets in continuous geometries}

Recalling von Neumann's \cite{neu3} algebraic construction of
continuous geometries,
let $F$ be one of the following fields with conjugation as involution:
$\IC$, $\IR$, and $\IF$, the real algebraic numbers.
Let the vector space $F^d$ be endowed with the canonical scalar
product and the
inner product spaces  $V_n(F)$ recursively defined by $V_0(F)=F^1$ and
  $V_{n+1}(F)=V_n(F) \oplus^\perp V_n(F)$.
For an anisotropic inner product space $V$
of $\dim V<\infty$, denote by $\lat(V)$
its MOL of subspaces.
Then  $\lat(V_n(F))$
 embeds into $\lat(V_{n+1}(F))$ via $U\mapsto U\oplus U$
and one may form the direct limit $\lat_\infty(F)$.
The latter admits a normalized dimension function
giving rise to a metric.
The MOL
 ${\sf CG}(F)$  is the metric completion of
$\lat_{\infty}(F)$ cf. \cite{hard}.
For a type II$_1$ von Neumann algebra factor $A$
let $\lat(A)$ denote the MOL of all projections
in $A$, cf. \cite{holl}.

\begin{theorem}
Let $L=\sf{CG}(\IC)$ or $L=\lat(A)$ for a type II$_1$ von Neumann
algebra factor. Then $\lat_\infty(\IF)$ embeds into
$L$ such that the image is a universal test set for $L$ 
and consist of elemens having rational normalized dimension with
denominator a power of $2$.
In particular, $\QL(L)=\QL(\lat_\infty(\IF))$.
\end{theorem} 
\begin{proof}
We begin with some observations.
 $\lat(V_n(\IF))$ is isomorphic to $\lat(\IF^{2^n})$.
$\QL(\lat(\IF^{2^n})) \subseteq \QL(\lat(\IF^m))$ if $m\leq 2^n$ 
since $\lat(\IF^m)$ is isomorphic to an interval of $\lat(\IF^{2^n})$.
$\QL(\lat(\IF^m))=\QL(\lat(\IR^m))$ since $\IF$ and  
$\IR$ are elementarily equivalent and since the ortholattices
can be interpreted within $\IF$ and $\IR$, respectively.
Finally, $\QL(\lat(\IR^m)) \subseteq \QL(\lat(\IC^k))$ 
 if $m=2k$ since $\lat(\IC^k)$ embeds into $\lat(\IR^m)$ 
considering $\IC^k$ the complexification of $\IR^k$. 
It follows that
 $\QL(\lat_\infty(\IF^n))\subseteq \bigcap_{n< \infty} \QL(\lat(\IC^n))$.

On the other hand one has $\bigcap_{n<\infty} \QL(\lat(\IC^n)) \subseteq \QL(L)$ by \cite{hard,hn}. Thus, given an embedding $\varepsilon:\lat_\infty(\IF^n))
\to L$, one can conclude  that
 $\QL(L)=\QL(\lat_\infty(\IF))$ whence ${\sf im} \varepsilon$
is a universal test set for $L$; moreover, the
normalized  dimension function $\delta$ of $L$
restricts to one on $\varepsilon(\lat(V_n(\IF)))$
and elements $x$
 of the latter have $\delta(x)=\frac{r}{2^n}$
with integer $r$, $0\leq r \leq 2^n$.

In case ${\sf CG}(\IF)$, the embedding $\varepsilon$
results from the embeddings $U \mapsto U\otimes \IC$
of $\lat(V_n(\IF))$ into $\lat(V_n(\IC))$. It remains to
establish $\varepsilon$ in case $L=\lat(A)$.

Recall, e.g. from \cite{berb2}, the notion of $*$-regular ring: A ring $R$ with
unit, endowed with an involution $x \mapsto x^*$
such that $xx^*=0$ only for $x=0$, and
such that for any $a$ there is $x$ with $a=axa$.
A $*$-regular ring $R$ admits a unique function $x \mapsto x^+$
satisfying the equations axiomatizing Moore-Penrose-Rickart
pseudo-inverse. In particular, 
the pseudo-inverse on a $*$-regular sub-$*$-ring $S$
is that inherited from $R$. 

For $F$ as aoove, the $m\times m$-matrices  
over $F$ form a $*$-regular ring $M_m(F)$ where $X^*$ is
the conjugate transpose. Observe that $M_m(F)$ embeds
into $M_{2m}(F)$ mapping $X$ onto the
block diagonal  matrix having $2$ diagonal blocks $X$.
The direct limit $M_\infty(F)$ of the $M_{2^n}(F)$, $n\to \infty$,
 is also $*$-regular
and the embeddings $M_{2^n}(\IF) \to M_{2^n}(\IC)$
result into an embedding $\iota:M_\infty(\IF)\to M_\infty(\IC)$.

The projections of a $*$-regular ring $R$ form a
MOL $\lat(R)$, ordered by $e \leq f \Leftrightarrow fe=e$,
 the fundamental  operations of which
are defined by terms in the language of $*$-rings with
pseudo-inversion. Thus, any embedding $\iota:S\to R$
 of $*$-regular ring restricts to an embedding
$\lat(S) \to \lat(R)$. Therefore, the
isomorphisms $\lat(\IF^{2^n})\to \lat(M_{2^n}(\IF)$ 
result into an isomorphism $\omega:\lat_\infty(\IF) \to 
\lat(M_\infty(\IF))$.

Recall from \cite[Definition 4.1.1]{mn} the notion of  \emph{approximately finite}
type II$_1$ factor. All these are isomorphic
as $*$-rings \cite[Theorem XII]{mn} and each
type II$_1$ factor contains one of them
as sub-$*$-ring \cite[Theorem XIII]{mn}. 
Finally,
by \cite[Lemma 4.1.1]{mn}, the type $[2^n, n \to \infty]$
approximately finite type II$_1$ factor
contains a  sub-$*$-ring isomorphic to  $M_\infty(\IC)$.
Thus any type II$_1$ factor $A$ contains a 
sub-$*$-ring  isomorphic to $M_\infty(\IF)$.

On the other hand, according to \cite[Theorem XV]{mn1}
 any  type II$_1$ von Neumann algebra factor $A$ extends to
a $*$-ring ${\sf U}(A)$ having the same projections, that is
 $\lat(A)=\lat({\sf U}(A))$;  moreover,
${\sf U}(A)$ is  $*$-regular \cite[Part II, Ch.II, App.2(IV)]{neu}
and \cite[p.191]{neuc}; cf. \cite{berb} and \cite[Theorem 4.2]{hn}.
Now,  $M_\infty(\IF)$ embeds into $A$ whence into ${\sf U}(A)$
and it follows that $\lat_\infty(F)$ embeds into 
$\lat(A)$.
\end{proof}
To see that ${\sf CG}(\IC)$
is not isomorphic to any $\lat(A)$, $A$ a type II$_1$ factor,
observe that a lattice isomorphism would induce
a ring isomorphism of the ring constructed in \cite{neu2}
onto ${\sf U}(A)$, contradicting von Neumann'a result,
cf. the introduction to \cite{neu}.

\end{document}